\documentclass[review]{elsarticle}
\usepackage{amssymb,latexsym,amsmath}
\usepackage{graphicx,times}
\usepackage{amsthm}
\usepackage{dcolumn}
\usepackage[margin=1.0in]{geometry}
\newtheorem{thm}{Theorem}

\newtheorem{cor}[thm]{Corollary}
\newtheorem{Def}{Definition}[section]
\newtheorem{Ex}{Example}[section]
\numberwithin{equation}{section}
\journal{Journal of \LaTeX\ Templates}

\begin{document}

\begin{frontmatter}

\title{ Solving Volterra Integro-Differential Equations involving Delay: A New Higher Order Numerical Method}


\author{Aman Jhinga$^{a}$ }
\ead{jhinga.aman@gmail.com}
\author{Jayvant Patade$^{a,b}$ }
\ead{dr.jayvantpatade@gmail.com}
\author{Varsha Daftardar-Gejji$^{a*}$ }
\ead{vsgejji@gmail.com}
\address{$^{a}$Department of Mathematics, Savitribai Phule Pune University, Pune - 411007, India\\ $^{b}$Department of Mathematics, Jaysingpur College, Jaysingpur - 416101, India.}


\cortext[mycorrespondingauthor]{Corresponding author}
\begin{abstract}
The aim of the present paper is to introduce a new numerical method for solving nonlinear Volterra integro-differential equations involving delay.  We apply  trapezium rule to  the integral involved in the equation. Further, Daftardar-Gejji and Jafari method (DGJ) is employed to solve the implicit equation. Existence-uniqueness theorem is derived for solutions of such equations and the error and convergence analysis of the proposed method is presented. We illustrate efficacy of the newly proposed method by constructing examples.
\end{abstract}
\end{frontmatter}

\section{Introduction}
During the last few decades, Volterra integro-differential equations (VIDEs) are widely used for mathematical modeling of various physical and biological phenomena. In this paper we  deal with VIDEs that incorporate memory effect. Such  VIDEs are useful in mathematical modeling of hereditary phenomena. The study of  Volterra delay-integro-differential equations (VDIDEs) has been an active area of research. It is found that VDIDEs are more effective than standard VIDEs in modeling the real-life phenomena. Hence researchers have developed theoretical and numerical analysis of VDIDEs. For example, Brunner \cite{brunner1989numerical} has given a survey of some recent developments in the numerical treatment of VDIDEs. Stability and boundedness of the solutions are discussed by Cemil Tunc \cite{tuncc2016new}. Zhang \textit{et al.} \cite{zhang2006general} have discussed general linear methods for solving VDIDEs. Explicit and implicit  Runge-Kutta methods for solving neutral Volterra integro-differential equations with delay have been developed by Enright \textit{et al.} \cite{enright1997continuous}.  Readers may refer to \cite{shakourifar2008numerical,patade2015new,patade2020novel} for more details in this regard.
\par In the present paper we consider VDIDE of the following form:
\begin{equation}\label{vide}
u'(x) = g(x,u(x)) + \int_{x_0}^x K(x,t, u(t-\tau))dt, \quad u(x)=\phi(x) \text{ for } x \in [-\tau,x_0].
\end{equation}
We employ a decomposition method proposed by Daftardar-Gejji and Jafari \cite{daftardar2006iterative} to generate a new, accurate and fast numerical method for solving nonlinear VDIDEs. 

The paper is organized as follows: Preliminaries are given in section \ref{Pre}. A new numerical method is presented in Section \ref{nnm}. Analysis of this numerical method is given in Section \ref{analy}. Section \ref{example}  deals with different types of illustrative examples. Conclusions are summarized in Section \ref{concl}.

\section{Preliminaries}\label{Pre}
\subsection{Basic Definitions and Notations}
In this section, we discuss some basic definitions and results.
\begin{Def}\cite{linz1969method}\label{1.11}
	Let $u_i$  denote the approximation to the exact value $u(x_i)$ obtained by a given method with step-size $h$. Then\\
	\textbf{(i)} a method is said to be convergent if and only if
	\begin{equation}
	\lim\limits_{h\rightarrow 0} \mid u(x_i)-u_i\mid\rightarrow 0,\quad i=1,2\cdots N.\label{1.1}
	\end{equation}
	\textbf{(ii)} a method is said to be of order $m$ if $m$ is the largest number for which there exists a finite constant $ \textsc{C}$ such that
	\begin{equation}
	\mid u(x_i)-u_i\mid\leq \textsc{C}h^m,\quad i=1,2\cdots N.\label{1.2}
	\end{equation}
\end{Def}

\subsection{Daftardar-Gejji and Jafari Method}\label{djm}
A new iterative method was introduced by Daftardar-Gejji and Jafari (DGJ) (2006) \cite{daftardar2006iterative} for solving nonlinear functional equations of the form
\begin{equation}
u = g + L(u) + N(u),\label{2.1}
\end{equation}
where $g$ is a known function, $L$ and $N$ are linear and nonlinear operators respectively. DGJ  method provides  a  solution to Eq.(\ref{2.1}) in the form of a series of the form
\begin{equation}
u= \sum_{i=0}^\infty u_i, \label{2.2}
\end{equation}
where
\begin{eqnarray}
u_0 &= & g, \nonumber\\
u_{m+1} &=& L(u_m) + G_m,\quad m=0,1, 2, \cdots,\label{2.6}\\
\textrm{and}\quad G_m &=& N\left(\sum_{j=0}^m u_m\right)- N\left(\sum_{j=0}^{m-1} u_m\right), m\geq 1.\nonumber
\end{eqnarray}
The $k$-term approximate solution is given by
\begin{eqnarray}
u = \sum_{i=0}^{k-1} u_i
\end{eqnarray}
for suitable integer $k$. DGJ method has been employed by many researchers to develop new numerical methods \cite{jhinga2018new,jhinga2019new,kumar2020new} for solving differential equations.

\section{Numerical Method}\label{nnm}
In the present section we construct a new numerical method based on DGJ decomposition to solve Volterra delay-integro-differential equations (VDIDEs) of the following form:
\begin{align}
u'(x) &= g(x,u(x)) + \int_{x_0}^x K(x,t, u(t-\tau))dt, \quad u(x_0)=u_0, \label{vd} \\
u(x)& = \phi (x) \text{ for } x \in [-\tau, x_0]. \label{vd2}
\end{align}
Consider a uniform partition of the interval $[-\tau,X]$ with grid points $x_j=jh:j=-M,-M+1,\dots,-1,0,1,\dots,N$ where $M$ and $N$ are integers such that $N=X/h$ and $M=\tau/h.$ \\
Integrating  Eq. (\ref{vd}) from $x=x_j $ to $x=x_j+h$, we get
\begin{equation}
u({x_j+h}) = u(x_j) + \int_{x_j}^{x_j+h} g(x,u(x)) dx + \int_{x_j}^{x_j+h}  \int_{x_0}^x K(x,t, u(t-\tau))dt dx \label{2}
\end{equation}
Applying trapezium rule \cite{jainnumerical} to evaluate integrals on right of Eq. (\ref{2}), we obtain
\begin{eqnarray}
u({x_j+h}) &=& u(x_j) + \frac{h}{2} g(x_j, u_j) + \frac{h^2}{4} \big( K(x_j, x_0, u(x_0-\tau)) + K(x_j, x_j, u(x_j-\tau))\nonumber\\
&& + K(x_{j+1}, x_0, u(x_0-\tau))\big)+ \frac{h^2}{2}\bigg( \sum_{i=1}^{j-1} K(x_j, x_i, u(x_i-\tau)) \nonumber\\
&& +  \sum_{i=1}^j K(x_{j+1}, x_i, u(x_i-\tau))\bigg) + \frac{h}{2} g(x_{j+1}, u_{j+1})+ O(h^3)\nonumber\\
&&+ \frac{h^2}{4}  K(x_{j+1}, x_{j+1}, u(x_{j+1}-\tau)).\label{3}
\end{eqnarray}
If $u_j$ denotes approximation to $u(x_j)$, then approximate solution at $x=x_j$ is given by
\begin{eqnarray}
u_{j+1} &=& u_j + \frac{h}{2} g(x_j, u_j) + \frac{h^2}{4} \left( K(x_j, x_0, u(x_{-M})) + K(x_j, x_j, u(x_{j-M})) + K(x_{j+1}, x_0, u(x_{-M}))\right)\nonumber\\
&& + \frac{h^2}{2}\left( \sum_{i=1}^{j-1} K(x_j, x_i, u(x_{i-M})) +  \sum_{i=1}^j K(x_{j+1}, x_i, u(x_{i-M}))\right) + \frac{h}{2} g(x_{j+1}, u_{j+1}) \nonumber\\
&& + \frac{h^2}{4}  K(x_{j+1}, x_{j+1}, u(x_{j+1-M})).\label{4.1}
\end{eqnarray}
where the delay term is approximated as given below:
\begin{eqnarray}
u(x_j-\tau) &=& u(jh-Mh)=u((j-M)h)=u(x_{j-M}), j=0,1,\dots,N\\
\text{ and }u(x_j) &=& \phi(x_j), ~ ~ j=-M,-M+1,\dots,0.
\end{eqnarray}
Eq. (\ref{4.1}) is of the form (\ref{2.1}),
where
\begin{eqnarray*}
	u &=& u_{j+1},\nonumber\\
	g &=& u_j + \frac{h}{2} g(x_j, u_j) + \frac{h^2}{4} \left( K(x_j, x_0, u(x_{-M})) + K(x_j, x_j, u(x_{j-M})) + K(x_{j+1}, x_0, u(x_{-M}))\right)\nonumber\\
	&& + \frac{h^2}{2}\left( \sum_{i=1}^{j-1} K(x_j, x_i, u(x_{i-M})) +  \sum_{i=1}^j K(x_{j+1}, x_i, u(x_{i-M}))\right)  \nonumber\\
	&& + \frac{h^2}{4}  K(x_{j+1}, x_{j+1}, u(x_{j+1-M})),\\
	N(u) &=& \frac{h}{2} g(x_{j+1}, u_{j+1}).
\end{eqnarray*}
Applying DGJ method to obtain 3-term approximate solution of eq. (\ref{4.1}), we get 
\begin{eqnarray}
u_{j+1} &=& u_j + \frac{h}{2} g(x_j, u_j) + \frac{h^2}{4} \bigg( K(x_j, x_0, u_{-M}) + K(x_j, x_j, u_{j-M}) + K(x_{j+1}, x_0, u_{-M})\bigg)\nonumber\\
&& + \frac{h^2}{2}\bigg( \sum_{i=1}^{j-1} K(x_j, x_i, u_{i-M}) +  \sum_{i=1}^j K(x_{j+1}, x_i, u_{i-M})\bigg)+ \frac{h^2}{4}  K(x_{j+1}, x_{j+1}, u(x_{j+1-M}))\nonumber\\
&& + \frac{h}{2}g\bigg(x_{j+1},  u_j + \frac{h}{2} g(x_j, u_j) + \frac{h^2}{4} \bigg( K(x_j, x_0, u_{-M}) + K(x_j, x_j, u_{j-M}) + K(x_{j+1}, x_0, u_{-M})\bigg)\nonumber\\
&& + \frac{h^2}{2}\bigg( \sum_{i=1}^{j-1} K(x_j, x_i, u_{i-M}) +  \sum_{i=1}^j K(x_{j+1}, x_i, u_{i-M})\bigg)+ \frac{h^2}{4}  K(x_{j+1}, x_{j+1}, u(x_{j+1-M}))\nonumber\\
&&+ \frac{h}{2}g\bigg(x_{j+1},  u_j + \frac{h}{2} g(x_j, u_j) + \frac{h^2}{4} \bigg( K(x_j, x_0, u_{-M}) + K(x_j, x_j, u_{j-M}) + K(x_{j+1}, x_0, u_{-M})\bigg)\nonumber\\
&& + \frac{h^2}{2}\bigg( \sum_{i=1}^{j-1} K(x_j, x_i, u_{i-M}) +  \sum_{i=1}^j K(x_{j+1}, x_i, u_{i-M})\bigg)+ \frac{h^2}{4}  K(x_{j+1}, x_{j+1}, u(x_{j+1-M}))\bigg)\bigg).\nonumber \label{5}
\end{eqnarray}
If we set
\begin{eqnarray}
M_1 &=& u_j + \frac{h}{2} g(x_j, u_j) + \frac{h^2}{4} \bigg( K(x_j, x_0, u_{-M}) + K(x_j, x_j, u_{j-M}) + K(x_{j+1}, x_0, u_{-M})\bigg)\nonumber\\
&& + \frac{h^2}{2}\bigg( \sum_{i=1}^{j-1} K(x_j, x_i, u_{i-M}) +  \sum_{i=1}^j K(x_{j+1}, x_i, u_{i-M})\bigg)+ \frac{h^2}{4}  K(x_{j+1}, x_{j+1}, u(x_{j+1-M})),\nonumber\\ \label{6}
M_2 &=& M_1 +  \frac{h}{2} g(x_{j+1}, M_1), \nonumber\label{7}
\end{eqnarray}
then equation (\ref{5}) becomes
\begin{eqnarray}
u_{j+1}=M_1 + \frac{h}{2} g(x_{j+1}, M_2),\label{9}
\end{eqnarray}
which is our new numerical method (NNM) for solving VDIDEs of the form (\ref{vd})-(\ref{vd2}).

\section{Analysis of Numerical Method}\label{analy}
\subsection{Existence and Uniqueness Theorem}\label{existnce}
The following result is generalization of Theorem (1) in \cite{patade2020novel}.

\begin{thm}\label{11}
	Consider the Volterra integro-differential equation
	\begin{eqnarray}
	u'(x) &=& g(x,u(x)) + \int_{x_0}^x K(x, t,u(t-\tau)) dt,\label{1}\\
	u(x) &=& \phi(x) ~ \text{ for } x \in [-\tau,x_0].\nonumber
	\end{eqnarray}
	Assume that  $g$ and $K$ are continuous and satisfy Lipschitz condition
	\begin{eqnarray}
	\parallel g(x, u_1)-g(x, u_2)\parallel &\leq& L_1 \parallel u_1-u_2 \parallel \\
	\parallel K(x, t, u_1(t-\tau))-K(x, t, u_2(t-\tau))\parallel &\leq& L_2 \parallel u_1(t-\tau)-u_2(t-\tau) \parallel
	\end{eqnarray}
	for every $\mid x-x_0\mid \leq a,  \mid t-x_0\mid \leq a, \parallel u_1 \parallel < \infty, \parallel u_2 \parallel < \infty\quad \textrm{and}\quad a>0$.
	Then the IVP (\ref{1}) has unique solution.
\end{thm}
\begin{proof} Integrating eq. (\ref{1}) and using initial condition, we get
\begin{eqnarray*}
	u(x) &=& u_0+\int_{x_0}^{x}g(z,u(z))dz+\int_{x_0}^{x}\int_{x_0}^{z}K(z,t,u(t-\tau))dt dz\\
	&=& u_0+\int_{x_0}^{x}\bigg(g(z,u(z))+\int_{x_0}^{z}K(z,t,u(t-\tau))\bigg)dt dz\\
	&=& u_0+\int_{x_0}^{x}G(z,u(z),u(z-\tau))dz,
\end{eqnarray*}
where $G(z,u(z),u(z-\tau))=g(z,u(z))+\displaystyle\int_{x_0}^{z}K(z,t,u(t-\tau))dt.$ We have the following observations:
\begin{enumerate}
	\item $u_0$ is continuous because it is constant,
	\item Kernel G is continuous for $0\leq x \leq a,$ because $g$ and $K$ are continuous in the same domain,
	\item G satisfies Lipschitz condition:
	\begin{equation*}
	\begin{split}
	\parallel G(x,u_1(x),&u_1(x-\tau))-G(x,u_2(x),u_2(x-\tau))\parallel=\parallel g(x,u_1(x))\\
	&+\displaystyle\int_{x_0}^{x}K(x,t,u_1(t-\tau))dt-g(x,u_2(x))-\int_{x_0}^{x}K(x,t,u_2(t-\tau))dt\parallel\\
	&\leq \parallel g(x,u_1(x))-g(x,u_2(x))\parallel+\parallel\displaystyle\int_{x_0}^{x}K(x,t,u_1(t-\tau))dt\\
	&+\int_{x_0}^{x}K(x,t,u_2(t-\tau))dt\parallel\\
	&\leq L_1 \parallel u_1-u_2 \parallel + a L_2 \parallel u_1(x-\tau)- u_2(x-\tau) \parallel ~ ~ ~ (\because \|x-x_0\|\leq a)\\
	\end{split}
	\end{equation*}
\end{enumerate}
Thus all the conditions of Theorem (1) in \cite{patade2020novel} are satisfied. Hence, The IVP (\ref{1}) has a unique solution.
\end{proof}
\begin{thm}\label{1.6}
	Assume that $g\in C[I\times\mathbb{R}^n,\mathbb{R}^n]$, $K\in C[I\times I\times\mathbb{R}^n,\mathbb{R}^n]$ and $\int_{s}^x \mid K(t,s,u(s))\mid dt \leq N$, for $x_0\leq s \leq x \leq x_0 + a$, $u\in \Omega = \{\phi\in C[I,\mathbb{R}^n]: \phi(x_0) = x_0 \  and \  \mid\phi(x) - u_0\mid \leq b\}$. Then IVP (\ref{1}) possesses at least one solution.
\end{thm}
\begin{proof} One can prove this result in a similar manner as given in \cite{lakshmikantham1995theory}.
\end{proof}
\subsection{Error Analysis }
\begin{thm}\label{13}
	Let $g$ and $K$ satisfy Lipschitz condition in second and third variables with Lipschitz constants $L_1$ and $L_2$ respectively and $g$ is bounded by $M, ~ M>0$. Then the numerical method (\ref{9}) is of third order.
\end{thm}
\begin{proof}
Suppose  $u_{j+1}$ is an approximation to $u(x_{j+1})$. By using Eq.(\ref{3}) and (\ref{9}), we obtain
\begin{eqnarray*}
	\mid u(x_{j+1})-u_{j+1}\mid &=& \mid\frac{h}{2} g(x_{j+1}, u_{j+1})-\frac{h}{2} g(x_{j+1}, M_2)\mid + O(h^3)\\
	&\leq& \frac{h}{2}L_1 \mid u_{j+1}-M_2\mid + O(h^3).
\end{eqnarray*}
Using Eq.(\ref{7}) and (\ref{9}), we get
\begin{eqnarray*}
	\mid u(x_{j+1})-u_{j+1}\mid &\leq& \frac{h}{2}L_1 \mid \frac{h}{2} g(x_{j+1}, M_2) - \frac{h}{2} g(x_{j+1}, M_1)\mid+ O(h^3)\\
	&\leq& \frac{h^2}{4}L_1 \mid g(x_{j+1}, M_2) - g(x_{j+1}, M_1)\mid+ O(h^3)\\
	&\leq& \frac{h^2}{4}L_1 \mid M_2 - M_1\mid+ O(h^3)\\
	&\leq& \frac{h^2}{4}L_1 \bigg(\frac{h}{2}\mid g(x_{j+1}, M_1)\mid\bigg)+ O(h^3)\\
	&\leq& h^3 \bigg(\frac{L_1}{8}\mid g(x_{j+1}, M_1)\mid\bigg)+ O(h^3) \leq h^3 \bigg(\frac{L_1}{8}M\bigg)+ O(h^3)\\
	& \leq & Ch^3,
\end{eqnarray*}
where $C$ is some constant.
$\Rightarrow$ The numerical method (\ref{9}) is of third order.
\end{proof}
\begin{cor}\label{14}
	The numerical method (\ref{9}) is convergent.
\end{cor}
\begin{proof} By Theorem (\ref{13}) and definition (\ref{1.11}), the  numerical method (\ref{9}) is convergent.
\end{proof}	
\section{Illustrative examples}\label{example}
\begin{Ex}\label{ex1}
	Consider the linear VIDE with delay
	\begin{align*}
	u'(x) &= e^{-1}(1-e^{x}) +  u(x) + \int_{0}^x u(t-1)dt,\quad u(0) = 1, \quad 0\le x\le 1, \\
	u(x) &=e^{x}, \quad x<0.
	\end{align*}
\end{Ex}
The exact solution is $	u(x) = e^{x}$.

\begin{center}
	\textbf{Table 1} Absolute errors and CPU times used for different values of  $h$\\	
	\begin{tabular}{cccc}
		\hline
		$x$ & Ab. error for  $h=0.01$ & Ab. error for  $h=0.02$ &  Ab. error for  $h=0.1$\\
		& CPU time  $0.0523884$ Sec. & CPU time $0.0139188$ Sec.& CPU time  $0.0021009$ Sec.\\
		\hline
		0.1 & $1.98103\times10^{-5}$& $7.88502\times10^{-5}$ & $1.89658\times10^{-3}$ \\
	
		0.2 & $2.15173\times10^{-5}$ & $8.5647\times10^{-5}$ & $2.06066\times10^{-3}$ \\
	
		0.3 & $2.34488\times10^{-5}$ & $9.33384\times10^{-5}$ & $2.24651\times10^{-3}$\\
	
		0.4 & $2.5633\times10^{-5}$ &  $1.02037\times10^{-4}$ & $2.45688\times10^{-3}$ \\
		
		0.5 & $2.81019\times10^{-5}$ & $1.11871\times10^{-4}$ & $2.69489\times10^{-3}$ \\
	
		0.6 & $3.08911\times10^{-5}$ & $1.22981\times10^{-4}$ & $2.96401\times10^{-3}$\\
	
		0.7 & $3.40406\times10^{-5}$ & $1.35527\times10^{-4}$ & $3.26816\times10^{-3}$ \\
		
		0.8 & $3.75955\times10^{-5}$ & $1.4969\times10^{-4}$ & $3.61174\times10^{-3}$  \\
		
		0.9 & $4.16061\times10^{-5}$ &  $1.65669\times10^{-4}$ & $3.99966\times10^{-3}$  \\
	
		1 & $4.6129\times10^{-5}$ & $1.83692\times10^{-4}$ & $4.43746\times10^{-3}$\\
		\hline
	\end{tabular}\\
\end{center}

\begin{tabular}{ll}
	\includegraphics[scale=0.55]{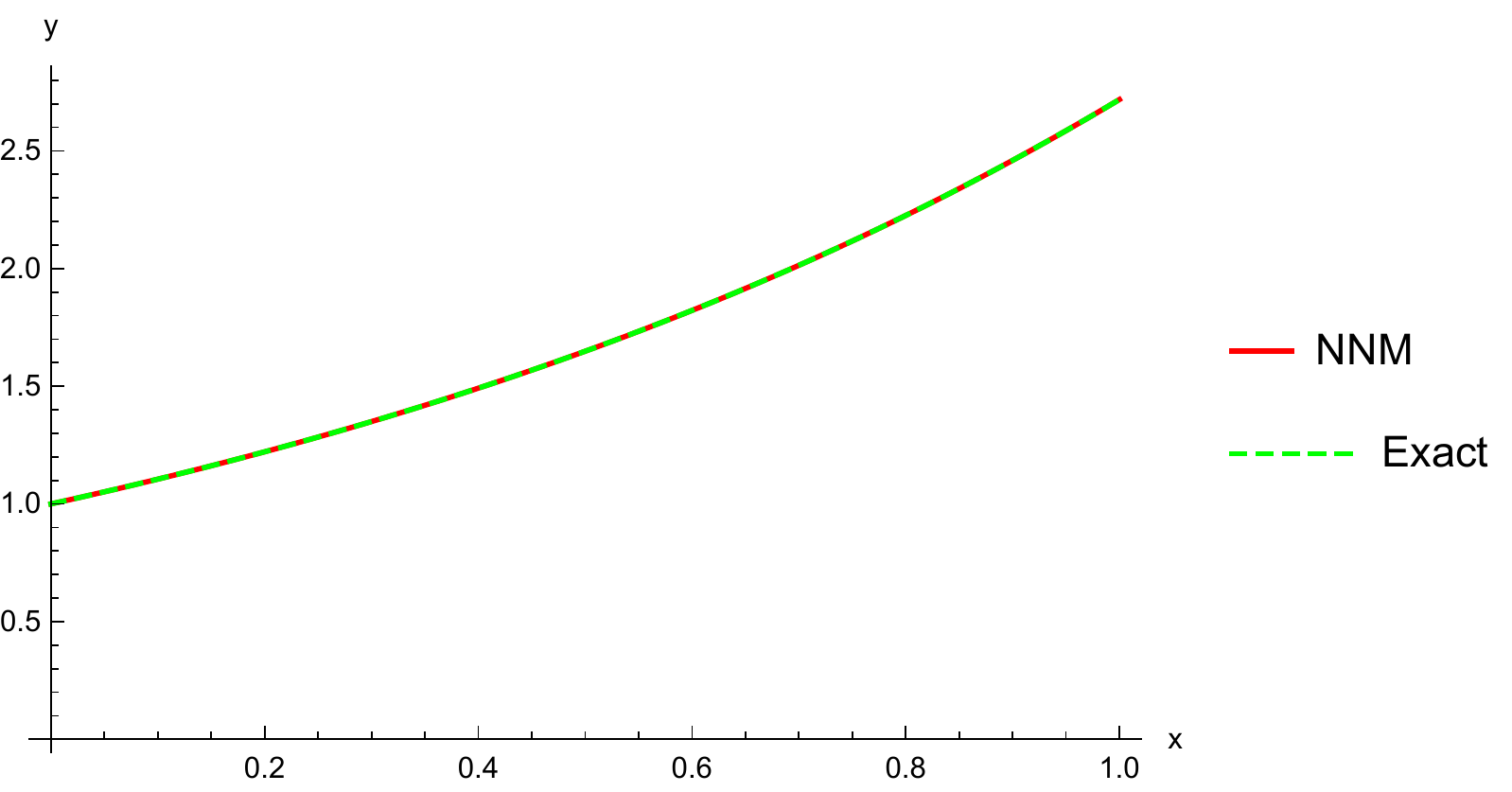} & \includegraphics[scale=0.55]{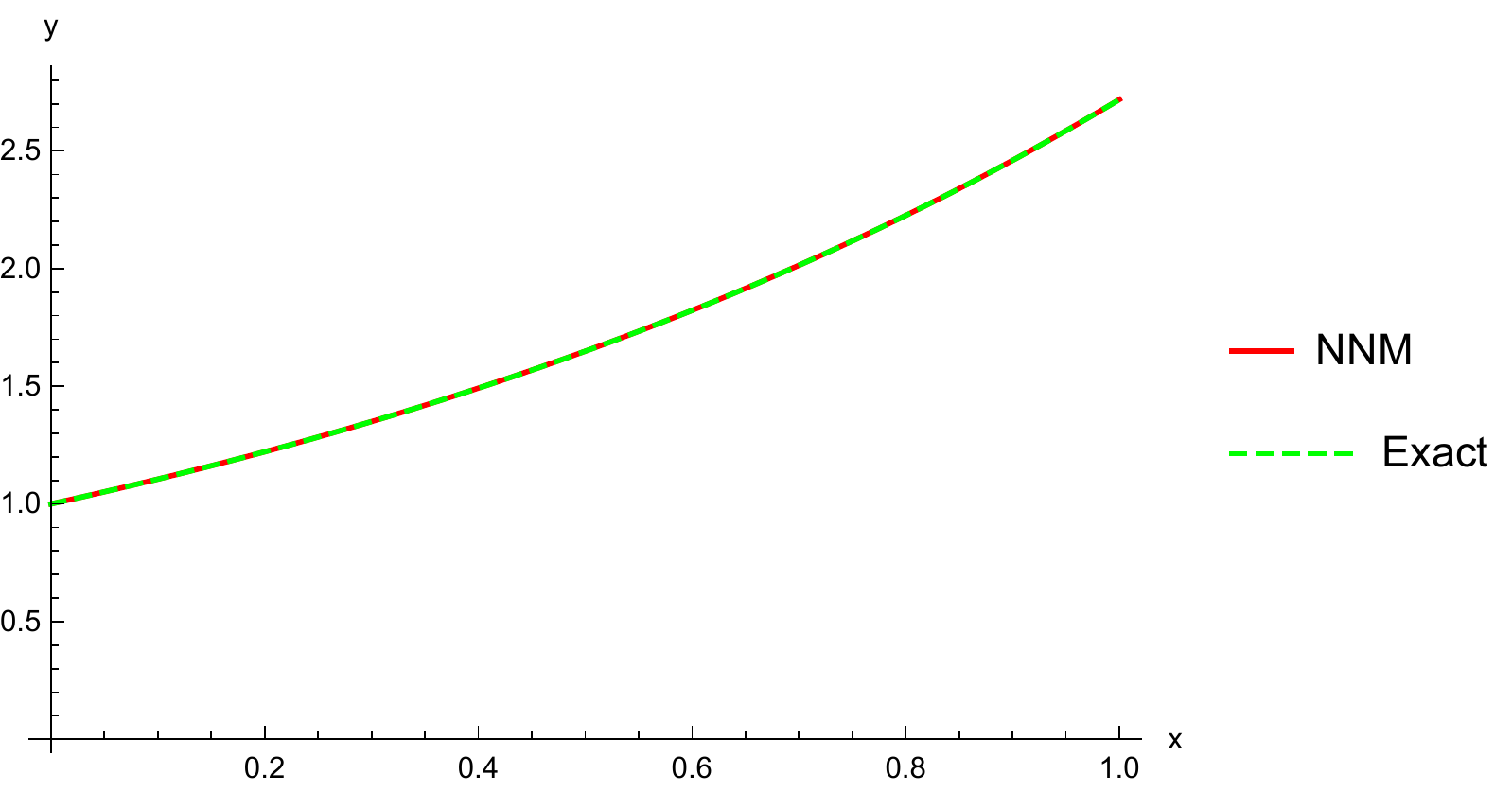}\\
	Fig.1: Comparison of solutions $u(x)$ of (\ref{ex1}) for $h=0.01$. & Fig.2: Comparison of solutions $u(x)$ of (\ref{ex1}) for $h=0.02$.\\
\end{tabular}

\begin{center}
	\begin{tabular}{c}
		\includegraphics[scale=0.55]{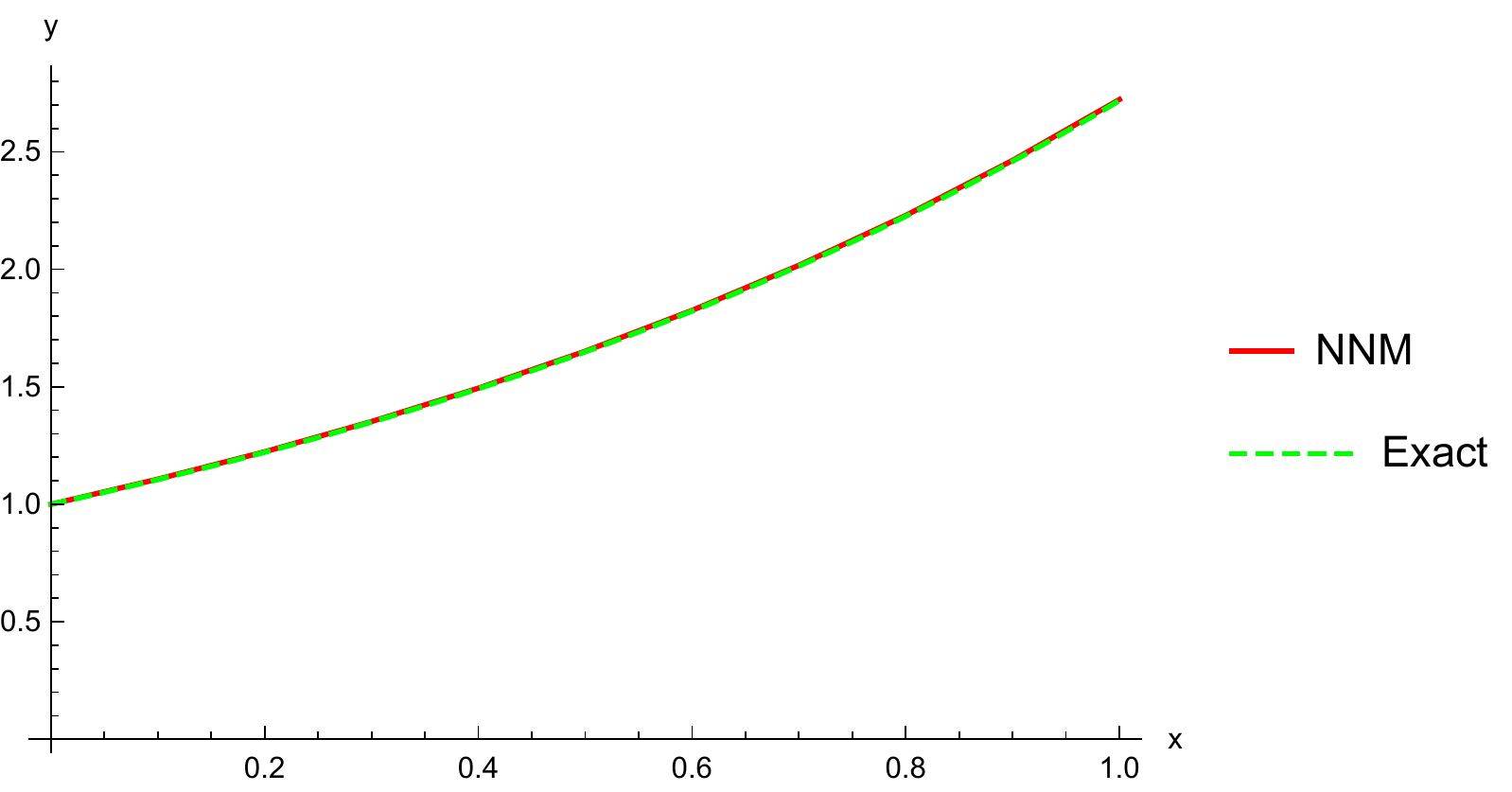}\\
		Fig.3: Comparison of solutions $u(x)$ of (\ref{ex1}) for $h=0.1$.\\
	\end{tabular}
\end{center}
 We compare our solution with exact solution for different values of $h$ in Figs.1-3 and the absolute errors, CPU time in Table 1. It is observed that NNM solution coincides with exact solution and the proposed method is time efficient.

\begin{Ex}\label{ex2}
	Consider the non-linear VIDE with delay
	\begin{align*}
	u'(x) &= -e^{x}\sinh x +  u(x) + \int_{0}^x u^2(t-1)dt,\quad u(0) = 1, \quad 0\le x\le 1, \\
	u(x) &=e^{x+1}, \quad x<0.
	\end{align*}
\end{Ex}
The exact solution is $	u(x) = e^{x+1}$.

\begin{center}
	\textbf{Table 2} Absolute errors and CPU times used for different values of  $h$\\	
	\begin{tabular}{cccc}
		\hline
		$x$ & Ab. error for  $h=0.01$ & Ab. error for  $h=0.02$ &  Ab. error for  $h=0.1$\\
		& CPU time  $0.0551474$ Sec. & CPU time $0.0158238$ Sec.& CPU time  $0.001088$ Sec.\\
		\hline
		0.1 & $5.39924\times10^{-5}$& $2.1491\times10^{-4}$ & $5.17093\times10^{-3}$ \\
		
		0.2 & $5.91385\times10^{-5}$ & $2.35415\times10^{-4}$ & $5.66983\times10^{-3}$ \\
		
		0.3 & $6.54017\times10^{-5}$ & $2.60385\times10^{-4}$ & $6.27996\times10^{-3}$\\
		
		0.4 & $7.30438\times10^{-5}$ &  $2.90867\times10^{-4}$ & $7.02773\times10^{-3}$ \\
		
		0.5 & $8.2388\times10^{-5}$ & $3.28155\times10^{-4}$ & $7.94573\times10^{-3}$ \\
	
		0.6 & $9.38329\times10^{-5}$ & $3.73844\times10^{-4}$ & $9.07421\times10^{-3}$\\
		
		0.7 & $1.0787\times10^{-4}$ & $4.29901\times10^{-4}$ & $1.04628\times10^{-2}$ \\
	
		0.8 & $1.25104\times10^{-4}$ & $4.98748\times10^{-4}$ & $1.21727\times10^{-2}$  \\
	
		0.9 & $1.4628\times10^{-4}$ &  $5.83369\times10^{-4}$ & $1.42792\times10^{-2}$  \\
		
		1 & $1.72316\times10^{-4}$ & $6.87436\times10^{-4}$ & $1.68753\times10^{-2}$\\
		\hline
	\end{tabular}\\
\end{center}

\begin{tabular}{ll}
	\includegraphics[scale=0.55]{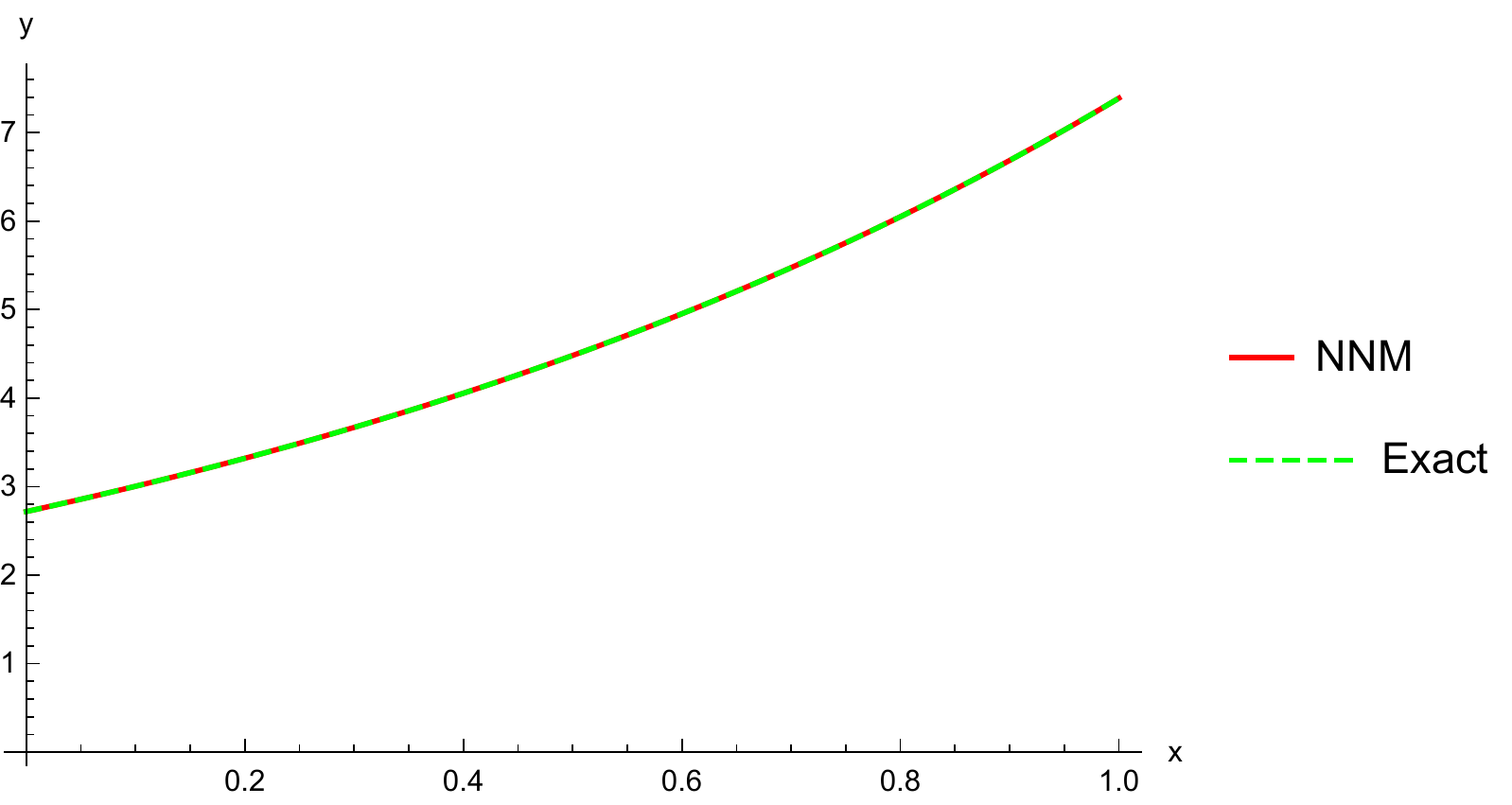} & \includegraphics[scale=0.55]{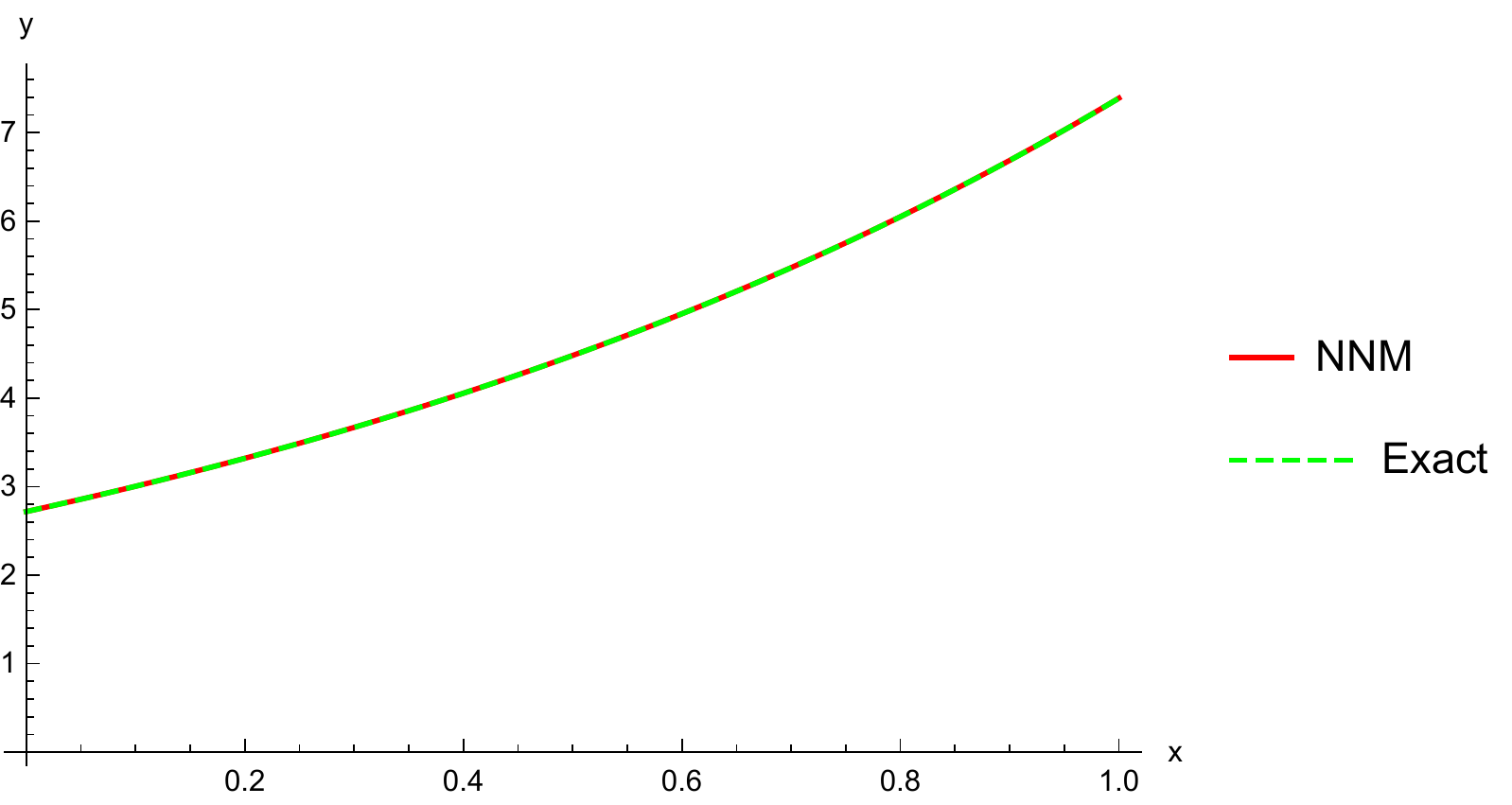}\\
	Fig.4: Comparison of solutions $u(x)$ of (\ref{ex2}) for $h=0.01$. & Fig.5: Comparison of solutions $u(x)$ of (\ref{ex2}) for $h=0.02$.\\
\end{tabular}

\begin{center}
	\begin{tabular}{c}
		\includegraphics[scale=0.55]{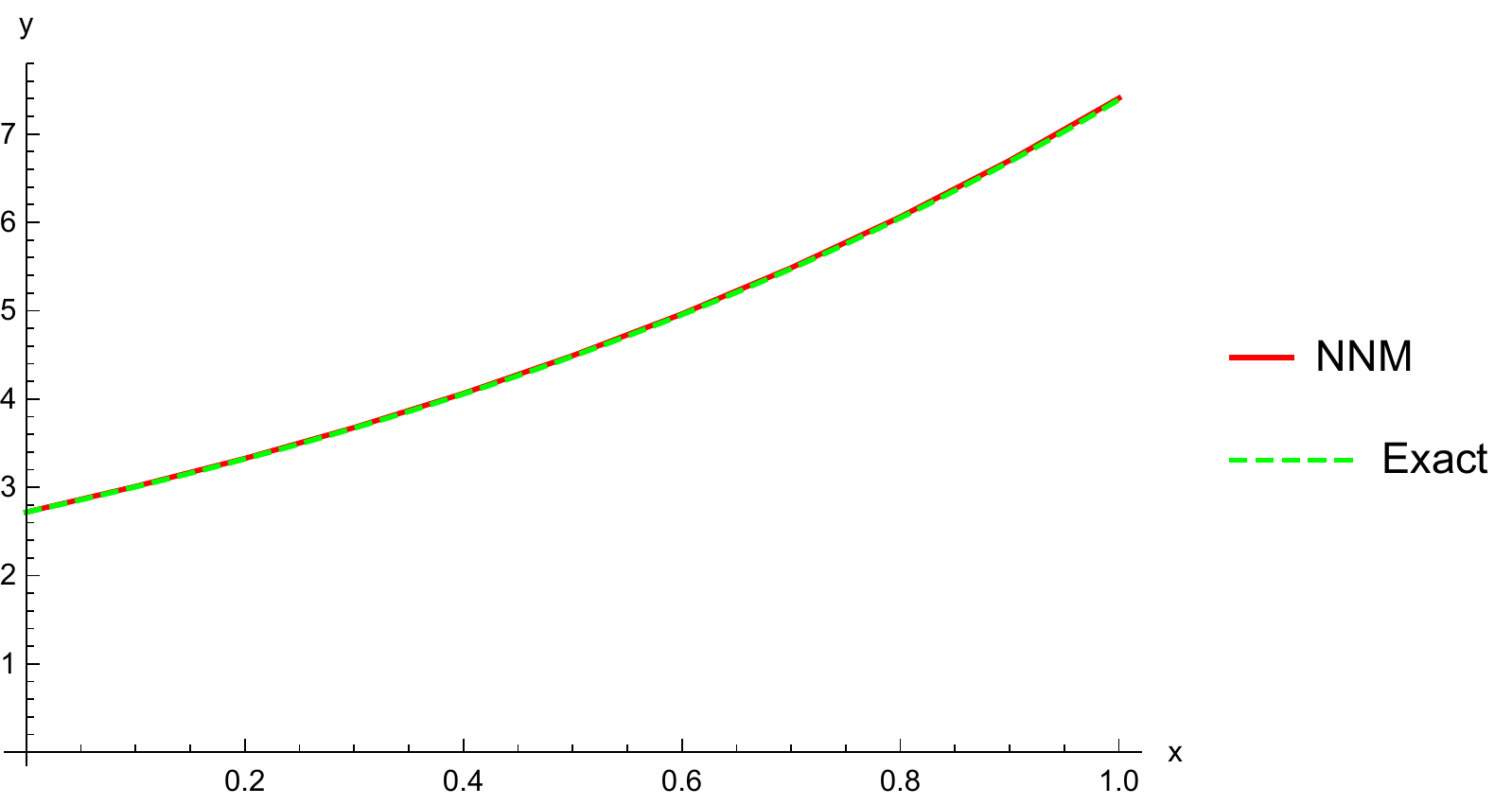}\\
	Fig.6: Comparison of solutions $u(x)$ of (\ref{ex2}) for $h=0.1$.\\
	\end{tabular}
\end{center}
 We compare our solution with exact solution for different values of $h$ in Fig.4-6. and the absolute errors, CPU time in Table 2. It is observed that NNM solution coincides with exact solution and the proposed method is time efficient.

\section{Conclusion}\label{concl}
A new numerical method is developed for solving nonlinear Volterra delay integro-differential equations (VDIDEs) of the form
\begin{equation}
u'(x) = g(x,u(x)) + \int_{x_0}^x K(x,t, u(t-\tau))dt, \quad u(x)=\phi(x) \text{ for } x \in [-\tau,x_0].
\end{equation}
Existence-uniqueness theorem is derived for solution of VDIDEs and  error and convergence analysis of the proposed method is presented. Efficiency of the proposed method is illustrated with various examples and it is observed that they are in very good agreement with the exact solutions.

\subsection*{Acknowledgements}
Aman Jhinga acknowledges Council of Scientific and Industrial Research, India for Senior Research Fellowship. Jayvant Patade acknowledges the Savitribai Phule Pune University, Pune, India for the postdoctoral fellowship [SPPU-PDF/ST/MA/2019/0001]. 

\section*{References}
\bibliography{references}

\begin{thebibliography}{10}
\expandafter\ifx\csname url\endcsname\relax
  \def\url#1{\texttt{#1}}\fi
\expandafter\ifx\csname urlprefix\endcsname\relax\def\urlprefix{URL }\fi
\expandafter\ifx\csname href\endcsname\relax
  \def\href#1#2{#2} \def\path#1{#1}\fi

\bibitem{brunner1989numerical}
H.~Brunner, The numerical treatment of volterra integro-differential equations
  with unbounded delay, Journal of computational and applied mathematics 28
  (1989) 5--23.

\bibitem{tuncc2016new}
C.~Tun{\c{c}}, New stability and boundedness results to volterra
  integro-differential equations with delay, Journal of the Egyptian
  Mathematical Society 24~(2) (2016) 210--213.

\bibitem{zhang2006general}
C.~Zhang, S.~Vandewalle, General linear methods for volterra
  integro-differential equations with memory, SIAM journal on scientific
  computing 27~(6) (2006) 2010--2031.

\bibitem{enright1997continuous}
W.~Enright, M.~Hu, Continuous runge-kutta methods for neutral volterra
  integro-differential equations with delay, Applied Numerical Mathematics
  24~(2-3) (1997) 175--190.

\bibitem{shakourifar2008numerical}
M.~Shakourifar, M.~Dehghan, On the numerical solution of nonlinear systems of
  volterra integro-differential equations with delay arguments, Computing
  82~(4) (2008) 241.

\bibitem{patade2015new}
J.~Patade, S.~Bhalekar, A new numerical method based on daftardar-gejji and
  jafari technique for solving differential equations, World J Modell Simul
  11~(4) (2015) 256--271.

\bibitem{patade2020novel}
J.~Patade, S.~Bhalekar, A novel numerical method for solving volterra
  integro-differential equations, International Journal of Applied and
  Computational Mathematics 6~(1) (2020) 1--19.

\bibitem{daftardar2006iterative}
V.~Daftardar-Gejji, H.~Jafari, An iterative method for solving nonlinear
  functional equations, Journal of Mathematical Analysis and Applications
  316~(2) (2006) 753--763.

\bibitem{linz1969method}
P.~Linz, A method for solving nonlinear volterra integral equations of the
  second kind, Mathematics of Computation 23~(107) (1969) 595--599.

\bibitem{jhinga2018new}
A.~Jhinga, V.~Daftardar-Gejji, A new finite-difference predictor-corrector
  method for fractional differential equations, Applied Mathematics and
  Computation 336 (2018) 418--432.

\bibitem{jhinga2019new}
A.~Jhinga, V.~Daftardar-Gejji, A new numerical method for solving fractional
  delay differential equations, Computational and Applied Mathematics 38~(4)
  (2019) 166.

\bibitem{kumar2020new}
M.~Kumar, A.~Jhinga, V.~Daftardar-Gejji, New algorithm for solving non-linear
  functional equations, International Journal of Applied and Computational
  Mathematics 6~(2) (2020) 26.

\bibitem{jainnumerical}
M.~Jain, S.~Iyengar, R.~Jain, Numerical methods for scientific and engineering
  computation new age international, 2003, New Delhi.

\bibitem{lakshmikantham1995theory}
V.~Lakshmikantham, Theory of integro-differential equations, Vol.~1, CRC press,
  1995.

\end{thebibliography}
\bibliographystyle{acm.bst}

\end{document}